\documentclass[reqno]{amsart}
\usepackage{bbm}
\usepackage[margin=1in]{geometry}
\usepackage{amsmath}
\usepackage{amssymb}
\usepackage{mathtools}

\usepackage{stmaryrd}
\usepackage[mathscr]{euscript}
\usepackage{amsthm}
\usepackage{enumerate}
\usepackage[hidelinks]{hyperref}
\usepackage{caption}
\usepackage{xcolor}
\usepackage{comment}
\usepackage{appendix}
\usepackage[style=numeric]{biblatex}
\usepackage{cleveref}

\newtheorem{thm}{Theorem}[section]

\newtheorem{lem}[thm]{Lemma}
\newtheorem{prop}[thm]{Proposition}
\newtheorem{clm}[thm]{Claim}
\newtheorem{defn}[thm]{Definition}
\newtheorem{example}[thm]{Example}
\newtheorem{rem}[thm]{Remark}
\newtheorem{con}[thm]{Conjecture}
\newtheorem{mainthm}{Theorem}

\newcommand\bR{{\mathbb R}}
\newcommand\bS{{\mathbb S}}

\newcommand{\bZ}{\mathbb{Z}}
\newcommand{\bN}{\mathbb{N}}
\newcommand{\scF}{\mathcal{IF}}
\newcommand{\scH}{\mathscr{H}}

\newcommand\diam{{\mathrm{diam}}}
\newcommand{\MinA}{\mathrm{MinA}}

\newcommand{\bhemi}{\partial\mathbb{S}^n_+}
\newcommand{\hemi}{\mathbb{S}^n_+}
\newcommand{\Scal}{\mathrm{Scal}}
\makeatletter

\@namedef{subjclassname@2020}{\textup{2020} Mathematics Subject Classification}
\makeatother               

\title{Width Stability of Rotationally Symmetric Metrics}

\author{Hunter Stufflebeam}
\address{Hunter Stufflebeam: Department of Mathematics, University of Pennsylvania, Philadelphia, PA, USA.}
\email{hstuff(at)sas(dot)upenn(dot)edu}

\author{Paul Sweeney Jr.}
\address{Paul Sweeney Jr.: Universit\`a di Trento, Dipartimento di Matematica, via Sommarive 14, 38123 Povo di Trento, Italy}
\email{paul(dot)sweeneyjr(at)unitn(dot)it}

\begin{document}

\begin{abstract}
We prove a conjecture of Marques-Neves in \cite{S}, and several alternative formulations thereof, about the stability of the min-max width of three-spheres under the additional assumption of rotational symmetry. We can moreover extend our results to all dimensions $n\geqslant 3$. 
\end{abstract}
\subjclass[2020]{}

\maketitle
%--------------------------------------------

\vspace{-15pt}\section{Introduction}

A classical question in Riemannian geometry is how curvature controls the size and topology of a manifold. Typically, comparison and rigidity theorems exemplify this type of phenomenon.  In \cite{MN}, Marques and Neves proved such theorems for Riemannian 3-spheres under the presence of lower scalar curvature bounds and the existence of a minimal surfaces produced via min-max methods. Precisely, they show that if there is a Riemannian metric on the 3-sphere with positive Ricci curvature and scalar curvature greater than or equal to 6, then the so-called ``width'' of the metric is less than or equal to $4\pi$. Moreover, equality is attained iff the metric is isometric to the standard unit round metric on the 3-sphere.

One can naturally wonder what happens when the hypotheses of a rigidity theorem are perturbed---if a geometric object almost satisfies the hypotheses of a rigidity theorem, is the object close to an object exemplifying the rigidity statement? These types of questions are typically phrased as  ``stability'' problems. At the 2018 IAS Emerging Topics Workshop on {\emph{Scalar Curvature and Convergence}} \cite{S}, Marques and Neves conjectured the following stability theorem related to their rigidity theorem above:
\begin{con}\label{conjMN}
Fix $D,V<\infty$. Suppose $(\mathbb S^3, g_k)$, $k=1,2,\ldots$ are Riemannian 3-spheres which satisfy
\[
\mathrm{Scal}_{g_k}\geqslant 6(1-k^{-1}), \text{ }\mathrm{MinA}_{g_k}\geqslant 4\pi(1-k^{-1}), \text{ }\diam_{g_j}\left(\mathbb S^3\right) \leqslant D, \text{ and } \mathrm{Vol}_{g_k}^3\left(\mathbb S^3\right) \leqslant V. 
\] Then $(\mathbb S^3, g_k)$ converges in the Sormani--Wenger Intrinsic Flat sense to the round unit sphere $(\mathbb S^3,g_{rd})$.
\end{con}

In \Cref{conjMN}, the original condition on the width is replaced with a stronger condition on $\MinA$, which is defined for a general Riemannian $n$-manifold $\left(M^n,g\right)$ by 
\[
\MinA_g:=\inf\left\{\mathrm{Vol}_g^{n-1}(\Sigma):\Sigma \text{ }\text{is a closed minimal hypersurface in}\text{ }M\right\}.
\] In fact, since the width $W_g$ is achieved by the area of a minimal surface, it is always true that $W_g\geqslant\mathrm{MinA}_g$. \Cref{conjMN} also drops the assumption of $\mathrm{Ric}>0$; in the proof of Marques--Neves' rigidity theorem, $\mathrm{Ric}>0$ is only used to ensure that the manifold contains no stable minimal 2-spheres. By \cite[Appendix A]{MN}, we see that if the scalar curvature of a $3$-manifold is sufficiently close to 6 and $\mathrm{MinA}$ is sufficiently close to $4\pi$, then the manifold contains no stable minimal embedded surfaces. Lastly, the additionally imposed bounds on diameter and volume guarantee the existence of a subsequential limit to $\{(\mathbb S^3, g_k)\}$ in the Sormani--Wenger intrinsic flat ($\mathcal{IF}$) topology (first defined in \cite{SW}) via Wenger's compactness theorem \cite{W}. This means that the aim of \Cref{conjMN} as stated is the identification of such a subsequential limit. 

In this paper, we prove \Cref{conjMN} under the additional assumption of rotational symmetry, but without the assumption of a volume upper bound. We can even extend our results to all dimensions $n\geqslant 3$. In particular, our first result is the following, where $\mathrm{d}_{\mathcal{IF}}$ denotes the  $\mathcal{IF}$ distance: 
\begin{mainthm}\label{t: mainIF}
     Fix $\delta>0$, $D>0$. Let $(\bS^n,g)$ be a rotationally symmetric metric on the $n$-sphere. There exists an  $\varepsilon=\varepsilon(n,D,\delta)>0$ such that if
    \begin{itemize}
        \item $\mathrm{diam}_g(\mathbb S^n)\leqslant D$,
        \item $\mathrm{Scal}_g\geqslant 6(1-\varepsilon)^2$,
        \item $\mathrm{MinA}_g\geqslant\omega_{n-1}(1-\varepsilon)^{n-1}$, where $\omega_{n-1}$ is the volume of the standard unit round $(n-1)$-sphere,
    \end{itemize} then $\mathrm{d}_{\mathcal{VIF}}((\mathbb S^3, g), (\mathbb S^3, g_{rd})):=\mathrm{d}_{\mathcal{IF}}((\mathbb S^3, g), (\mathbb S^3, g_{rd}))+\lvert\mathrm{Vol}^n_g(\mathbb S^n)-\mathrm{Vol}^n_{g_{rd}}(\mathbb S^n)\rvert\leqslant\delta$.
\end{mainthm}

We remark that without the uniform lower bounds on $\MinA$ to prevent bubbling along the sequence as in \Cref{conjMN}, counterexamples can be constructed as shown by the second author in \cite{Swe}. These examples are rotationally symmetric and so, even with our added hypothesis of rotational symmetry here, the $\MinA$ condition is necessary. We also note that we do not need a volume bound for our proof, as it ends up being a \emph{conclusion} (see also \cite[Remark 1.5]{PTW}).

 In dimension two, M\'{a}ximo and the first author have proven an analogous stability theorem which says that strictly convex $2$-spheres, all of whose simple closed geodesics are close in length to $2\pi$, are $C^0$ Cheeger--Gromov close to the round sphere \cite{MS}. In \cite{BamlerMaximo}, Bamler and M\'aximo prove a version of \Cref{conjMN} without any symmetry assumption, but under the stronger curvature assumptions that $\mathrm{sec}>0$ and $\mathrm{Scal}\geqslant 6$ instead of solely $\mathrm{Scal}\geqslant 6$. Their result is also phrased in terms of the $C^0$ Cheeger--Gromov distance. It is interesting to note that despite the similarity of these results, all of our respective proof methods are quite different. The main techniques of \cite{MS} come from metric geometry and geometric measure theory, the techniques of \cite{BamlerMaximo} come from Ricci flow, and ours in this work principally come from the Sormani--Wenger Intrinsic Flat theory. 
 
 Furthermore, it is interesting to compare the width rigidity of Marques--Neves with Min-Oo's conjecture: if $g$ is a smooth metric on the hemisphere $\hemi$ such that the scalar curvature satisfies $\Scal_g\geq n(n-1)$, the induced metric on the boundary $\bhemi$ agrees with the standard unit round metric on $\bS^{n-1}$, and the boundary $\bhemi$ is totally geodesic with respect to $g$, then $g$ is isometric to the standard unit round metric on $\hemi$. Perhaps surprisingly, Min-Oo's conjecture turned out to be false in general with $n\geqslant 3$ due to an example of Brendle, Marques, and Neves \cite{BMN} (see also \cites{CEM,Swe2}). Nevertheless, many special cases of Min-Oo's Conjecture are known to be true---see eg. \cites{BM,HW, HW2, HLS}. In dimension $n=2$, Min-Oo's conjecture is in fact true on the nose, and is contained in an old theorem of Toponogov (see \cite{Klingenberg} Theorem 3.4.10). In \cite{St}, the first author has proven a stability version of this Min-Oo Conjecture/Theorem. 
 
 Of particular interest here is the work of Hang and Wang in \cite{HW}, wherein the authors show that: given a metric satisfying the hypotheses of Min-Oo's Conjecture which is additionally conformal to the standard unit round metric, then the conjecture holds. Using PDE techniques from the proof of this result, we are able to prove an analogous ``width rigidity theorem'' in all dimensions $n\geqslant 3$ akin to the Marques--Neves theorem. In particular, we define a notion of width $W_g$ for rotationally symmetric metrics in all dimensions $n\geqslant 3$ in Section \ref{minmax}, and show that  
 \begin{mainthm} \label{t: rigidity}
     Let $n\geqslant3$ and $g$ be a rotationally symmetric Riemannian metric on the n-sphere, $\bS^n$ such that
\begin{itemize}
    \item $\mathrm{Scal}_g\geqslant n(n-1)$,
    \item $W_g\geqslant\omega_{n-1}$. 
\end{itemize}
Then $(\mathbb S^n, g)$ is isometric to $(\mathbb S^n, g_{rd})$.
 \end{mainthm}

 Other stability theorems involving scalar curvature and rotational symmetry have been proven. Lee and Sormani investigated the intrinsic flat stability of the positive mass theorem \cite{LeS} and the Penrose inequality \cite{LeS2} under the assumption of rotational symmetry. More recently, Park, Tian, and Wang \cite{PTW} proved that given $A,D>0$ and a sequence of oriented rotationally symmetric Riemannian 3-manifolds without boundary $(M_j^3,g_j)$ such that $\diam_{g_j}\left(M_j\right)<D$, $\mathrm{Scal}_{g_j}\geq 0$, and $\mathrm{MinA}_{g_j}\geq A>0$, then there is a subsequence that converges to a metric space $(M_\infty,g_\infty)$ such that $\mathrm{Vol}^3_{g_{j_k}}(M_{j_k})$ converges to the mass $\mathbf{M}(M_\infty)$, and where $g_\infty$ is a $C^0$, $H^1$, rotationally symmetric metric. Moreover, in a certain sense $g_\infty$ has nonnegative generalized scalar curvature. Therefore, by this result, we already know that there is a subsequence in \Cref{conjMN} that $\scF$-converges to such a limit space. The novel point of the conclusion of \Cref{t: mainIF} is that the limit of such a sequence is the unit round sphere. We do not, however, appeal to the results of \cite{PTW} in our arguments here.

 Finally, we note that if we add a Ricci nonnegative hypothesis to \Cref{conjMN} then we can replace the $\MinA_g$ lower bound with a lower bound on our width $W_g$ and prove the following Gromov--Hausdorff stability theorem for rotationally symmetric metrics.

\begin{mainthm}\label{t: mainRicci}
    Fix $n\geqslant 3$ and $\delta>0$. There exists an $\varepsilon=\varepsilon(n, \delta)>0$ such that if $g$ is a rotationally symmetric metric on the $n$-sphere $\mathbb S^n$ satisfying 
    \begin{itemize}
        \item $\mathrm{Scal}_g\geqslant n(n-1)(1-\varepsilon)^2$;
        \item $\mathrm{Ric}_g\geqslant 0$;
        \item $W_g\geqslant \omega_{n-1}(1-\varepsilon)^{n-1},$ 
    \end{itemize} then $\mathrm{d}_{GH}((\mathbb S^n, g), (\mathbb S^n, g_{rd}))\leqslant\delta$.
\end{mainthm}

Lastly, as a result of the proofs of the above theorems we have a third version of the stability theorem in rotational symmetry---this time with respect to a convergence studied by Dong and Song in \cite[Theorem 1.3]{DS} (see also \cite{D}) to resolve a conjecture of Huisken and Ilmanen \cite[Conjecture page 430]{HI} about the stability of the positive mass theorem. Recently, a convergence of this type also was used by Bryden and Chen for stability theorems related to tori \cite{BC}. Roughly, the idea of this convergence is to surgically remove a controlled ``bad set'' from each manifold in the sequence, so that the remainders converge in the Gromov--Hausdorff sense to the desired limit. The notion of ``bad set" varies slightly from place to place in the literature, but the unifying idea is that the ``bad set" of non-convergent points is geometrically small. In particular, we show

\begin{mainthm}\label{t: mainDSHI}
     Fix $n\geqslant 3$, $D>0$, and $\delta>0$. There exists an $\varepsilon=\varepsilon(n, \delta, D)>0$ such that if $g$ is a rotationally symmetric metric on the $n$-sphere $\mathbb S^n$ satisfying 
    \begin{itemize}
        \item $\mathrm{diam}_g(\mathbb S^n)\leqslant D$;
        \item $\mathrm{Scal}_g\geqslant n(n-1)(1-\varepsilon)^2$;
        \item $\mathrm{MinA}_g\geqslant \omega_{n-1}(1-\varepsilon)^{n-1},$ 
    \end{itemize} then there exists a smooth domain $Z\subset\mathbb S^n$ with at most two connected components satisfying \[\mathrm{Vol}_g^n(Z)+\mathrm{Vol}^{n-1}_g(\partial Z)\leqslant \delta,\] so that $\mathrm{d}_{GH}((\mathbb S^n\setminus Z, g), (\mathbb S^n, g_{rd}))\leqslant\delta$.
\end{mainthm}

The following is an outline of the paper. In \Cref{sec: back}, we provide the necessary background and preliminaries on rotationally symmetric manifolds, min-max theory, Gromov--Hausdorff convergence, and Sormani--Wenger intrinsic flat convergence. In \Cref{sec: width rig}, we prove \Cref{t: rigidity}. In \Cref{sec: a priori}, we prove some preliminary propositions needed in the proof of the stability theorems. Finally, in \Cref{sec: stab}, we prove \Cref{t: mainIF}, \Cref{t: mainRicci}, and {\Cref{t: mainDSHI}}.

\subsection{A Comment on Notation}\label{notation}

Aside from other notation which we will introduce in the coming sections, we emphasize here that we follow tradition in using notation such as $\Psi=\Psi(x)=\Psi(x_1, x_2,\ldots : a_1,a_2,\ldots)$ to denote a non-negative function, which may change from line to line, depending on any number of variables $x_i$ and any number of parameters $a_i$ with the property that if the $a_i$ are all held fixed, $\Psi\searrow 0$ as $x_i\to 0$. 

Throughout the paper, we will also denote the Hausdorff metric by $\mathrm{d}_\mathcal H(\cdot,\cdot)$, the Gromov-Hausdorff (GH) distance by $\mathrm{d}_{GH}(\cdot,\cdot)$, the Sormani-Wenger Intrinsic Flat ($\mathcal{IF}$) distance by $\mathrm{d}_{\mathcal{IF}}(\cdot,\cdot)$, and the Volume Preserving Intrinsic Flat ($\mathcal{VIF}$) distance by $\mathrm{d}_{\mathcal{VIF}}(\cdot,\cdot)$. 

\section{Acknowledgments}
This project has received funding from the European Research Council (ERC) under the European Union’s Horizon 2020 research and innovation programme (grant agreement No. 947923). This paper was partially supported by Simons Foundation International, LTD. This work was supported in part by NSF Grant DMS-2104229 and NSF Grant DMS-2154613.

\section{Background and Preliminaries}\label{sec: back}

\subsection{Rotationally Symmetric Manifolds: An Overview}\label{overview}
Since our main objects of interest in this paper are rotationally symmetric metrics on compact manifolds, we will work in coordinates that are well suited to this particular situation. Namely, we will view such an $n$-dimensional $(M^n,g)$ as a warped product of a line segment $(0,D)$ with the round $(n-1)$-sphere, endowed with a metric
\[
g=ds^2+f(s)^2g_{rd}
\]
where $g_{rd}$ is the round metric on $\mathbb S^{n-1}$. In general, we will use $g_{rd}$ to denote the round metric on spheres of different dimensions, but which dimension is meant should always be clear from context. Here, $f(s)$ is a smooth, non-negative function on $(0,D)$. In particular, it is well known (see eg. \cite{P}) that if $f(0^+)=0$, then smoothness of the metric implies that (after smoothly extending $f$ to $s=0$): \begin{itemize}
    \item $f'(0)=1$ and $f^{(even)}(0)=0$;
    \item $f(s)>0$ if $s\in (0,D)$;
    \item if $f(D^-)=0$, then also $f'(D)=-1$ and $f^{(even)}(D)=0$. 
\end{itemize} We will be able to focus our attention in this paper on regions of $(M^n,g)$ where $f'(s)$ is nonzero and where a related coordinate system can also be utilized. In general, a rotationally symmetric manifold may be broken into parts based on the trichotomy 
\[
f'(s)<,=,>0,
\] 
where $\{f'(s)=0\}$ contains the cylindrical pieces of the manifold. On a connected component with $f'(s)\neq 0$, we may instead use $r:=f(s)$ itself as the coordinate, and consider the metric \[g=\frac{dr^2}{V(r)}+r^2g_{rd}\] for a positive smooth function $V(r)=(f'(f^{-1}(r)))^2$. We note that $V(r)$ satisfies its own appropriate derivative constraints at the endpoints of the interval, provided such a coordinate singularity point corresponds to a genuine smooth manifold point. 

Focusing our attention now on such regions where both coordinate systems are available, we record the formulas we'll need for the fundamental geometric objects of concern. Given a metric 
\begin{equation}\label{e: rotsymmetric}
    g=\frac{dr^2}{V(r)}+r^2g_{rd}=ds^2+f(s)^2g_{rd}
\end{equation}
we let $\Sigma_s:=\{s\}\times\mathbb S^{n-1}$, endowed with its induced metric. We then compute (see eg. \cite{L}):  

\begin{description}
\item[Ricci Curvature]\label{ricciformula} 
    \begin{align*}
        \mathrm{Ric}_g(\partial_s,\partial_s)=-(n-1)\frac{f''(s)}{f(s)}
    \end{align*}
\item[Scalar Curvature]\label{scalarformula} 
    \begin{align*}
        \mathrm{Scal}_g&=\frac{n-1}{r^2}\left[(n-2)(1-V(r))-rV'(r)\right]\\
        &=(n-1)(n-2)\left(\frac{1-|f'(s)|^2}{f(s)^2}-\frac{2}{n-2}\frac{f''(s)}{f(s)}\right)
    \end{align*} 
\item[Mean Curvature of $\Sigma_s$]\footnote{In the following, the mean curvature is the divergence of the normal vector field along $\Sigma_s$ pointing in the direction of $\partial_s$.}\label{mcformula} 
    \begin{align*} 
        H_g(\Sigma_r)&=(n-1)\frac{V(r)^{1/2}}{r}\\ 
        &=(n-1)\frac{f'(s)}{f(s)}
    \end{align*}
\item[Second Fundamental Form of $\Sigma_s$]\label{Aformula} 
    \begin{align*} 
        A_{ij}&=f(s)f'(s)\delta_{ij}
    \end{align*}
\item[Volume Form]\label{volformula}  
\begin{align*}d\mathrm{Vol}^n_g=\frac{r^{n-1}}{V(r)^{1/2}}\mathrm{d}\mathcal{L}^1(r)\otimes d\mathrm{Vol}^{n-1}_{g_{rd}}\end{align*}
\end{description}

We also record, for the reader's convenience, that when $g=g_{rd}$ is the round metric on an $n$-hemisphere, the coordinate representation in \eqref{e: rotsymmetric} has \begin{align*}\begin{cases} 
V_{rd}(t)=1-t^2 &\text{ on }\quad [0,1] \\
f_{rd}(s)=\sin(s) &\text{ on }\quad [0,\pi/2].\end{cases}\end{align*}

Lastly, we recall the following a priori estimate for $f'(s)$ when the scalar curvature is non-negative, whose proof is a simple application of the mean value theorem and the ODE for $\mathrm{Scal}_g$: 

\begin{lem}[\emph{cf. {\cite[Lemma 2.6]{PTW}}}]\label{ScalarLipschitz} If $\mathrm{Scal}_g\geqslant 0$, then $|f'(s)|\leqslant 1$ everywhere. Consequently, $0\leqslant V(r)\leqslant 1$. 
\end{lem}

\subsection{Min-Max Theory}\label{minmax}

In this section we introduce the main invariant of focus -- the Simon-Smith min-max width $W^{SS}_g$ of the metric $g$ (and a high dimensional analogue for rotationally symmetric metrics). 

The modern min-max theory for constructing minimal surfaces in manifolds has its genesis in the work of Birkhoff and Lyusternik-Schnirelmann on the existence of simple closed geodesics in two-spheres \cites{BirkhoffGeodesic, LyusternikSchnirelmann}, and can be described succinctly as an extension of classical Morse theory to the area functional. Since the pioneering work of Almgren and Pitts in \cites{AlmgrenHomotopyGroups, Pitts1974, PittsThesis}, min-max theory has been at the center of a veritable industry in geometric analysis, whose rich history is impossible to fully survey here. Amongst the many groundbreaking results in the area, we highlight the work of Marques--Neves \cite{MNInfinite}, Irie--Marques--Neves \cite{IrieMarquesNeves}, and Song \cite{SongInfinite} in the resolution of Yau's conjecture in \cite{YauSeminar} about the existence of infinitely many closed, embedded, minimal surfaces in every closed $n$-manifold ($3\leqslant n\leqslant 7$), and the work \cite{MNWillmore} of Marques--Neves in the resolution of the Willmore Conjecture. 

We now turn to introducing the aspects of min-max theory relevant to us here -- those of the \emph{Simon--Smith} variant of the theory originally pioneered by Almgren--Pitts. This min-max theory, developed by Smith in \cite{SmithThesis} (see also \cite{ColdingDeLellis} for a fantastic accounting), produces a smooth and embedded minimal hypersurface in a Riemannian 3-manifold $(M^3,g)$. Consider here a Riemannian 3-sphere $(\mathbb S^3, g)$. The starting point is the construction of \emph{sweepouts}: one-parameter families of 2-spheres in $(\mathbb S^3, g)$ starting and ending at degenerate point spheres, so that the induced maps $F:\mathbb S^3\to\mathbb S^3$ have nonzero degree. The $\emph{Simon--Smith width}$ of $(\mathbb S^3, g)$ is defined to be the infimum over the areas of largest two-spheres in all such one-parameter family of two-spheres ``sweeping out'' $(\mathbb S^3, g)$. Let us give the following more precise definition (cf. \cites{ColdingDeLellis, MN}):

\begin{defn}
    Given the standard embedding of $\mathbb S^3\hookrightarrow\mathbb R^4$, consider the level sets $\overline\Sigma_t:=(x^4)^{-1}(t)$, $t\in [-1,1]$, of the coordinate function $x^4:\mathbb S^3\subset\mathbb R^4\to\mathbb R$ (more generally, one could directly work with the level sets of a given Morse function on $(\mathbb S^3,g)$). Let also $\overline\Lambda$ be the collection of all families $\left\{\Sigma_t\right\}_{t\in[-1,1]}$ with the property that $\Sigma_t=F_t(\overline\Sigma_t)$ for some smooth one-parameter family of diffeomorphisms $F_t:\mathbb S^3\to\mathbb S^3$ all smoothly isotopic to the identity $\mathrm{id}:\mathbb S^3\to\mathbb S^3$. The \emph{Simon--Smith width} of $(\mathbb S^3, g)$ is then defined to be the number \[W^{SS}_g:=\inf_{\left\{\Sigma_t\right\}\in\overline\Lambda}\left\{\sup_{t\in [-1,1]}\mathrm{Vol}^2_g\left(\Sigma_t\right)\right\}.\]
\end{defn}

One of the touchstone theorems in the Simon--Smith theory, which motivates our work in this paper, is the following result. In it, the area of the minimal sphere produced by min-max theory realizing the value $W^{SS}_g$ can be viewed as a \emph{size invariant} of $(\mathbb S^3, g)$. This is akin to the way in which the lengths of closed geodesics indicate the size and shape of positively curved two-spheres (see, for example, \cites{MS,Croke}, and Chapter 3 of \cite{Klingenberg}). 

\begin{thm}[Marques--Neves cf. \cite{MN}]\label{MN}
    Let $(\mathbb S^3, g)$ have $\mathrm{Ric}(g)>0$ and $\mathrm{Scal}_g\geqslant 6$. Then there exists an embedded minimal $\Sigma\simeq\mathbb S^2$ of Morse index one\footnote{Recall that the \emph{Morse index} of a two-sided minimal hypersurface $\Sigma^{n-1}$ in $(M^n,g)$ is the number of negative eigenvalues (counted with multiplicity) of the Jacobi operator $L^\Sigma=\Delta^\Sigma+\mathrm{Ric}_g(\nu,\nu)+|A_\Sigma|^2$, where $\Delta^\Sigma=\mathrm{div}^\Sigma\circ\nabla^\Sigma$, and $\nu$ is a unit normal to $\Sigma$.} such that \[W^{SS}_g=\mathrm{Vol}^2_g(\Sigma)\leqslant 4\pi,\] with equality possible iff $(\mathbb S^3, g)$ is isometric to the unit round sphere $(\mathbb S^3, g_{rd})$. 
\end{thm}

The collection $\overline\Lambda$ is commonly referred to as the \emph{saturated family of sweepouts generated by $\left\{\overline\Sigma_t\right\}$}. In the rotationally symmetric case, it is easy to see that $\overline{\Lambda}$ is also generated by the following sweepout, which we term the \emph{canonical sweepout}. It will be the focus of essentially all of our computations in the sequel.

\begin{defn}[The Canonical Sweepout]
    Let \[g=ds^2+f(s)^2g_{rd} \text{\quad with }s\in[0,D]\] be a smooth, rotationally symmetric metric on the three-sphere $\mathbb S^3$. Let $p_-$ be the ``south pole'' of the suspension, where $s=0$. The canonical sweepout is the one-parameter family $\{\Sigma_s\}_{s\in [0, D]}$ of two spheres, where $\Sigma_s:=(\mathrm{dist}_{p_-})^{-1}(s)$. Each $\Sigma_s$ has the induced metric $f(s)^2g_{rd}$, and therefore $\mathrm{Vol}^2_g(\Sigma_s)=4\pi f(s)^2$.  
\end{defn}

Of course, our theorems are stated for all dimensions $n\geqslant 3$, a generality ultimately afforded in return for the rotational symmetry assumption. However, in high dimensions $n\geqslant 4$ \emph{there is no Simon--Smith theory available}. In part, this is due to the fact that the regularity theory for minimal hypersurfaces produced from isotopy classes, which forms a key part of the Simon-Smith method, cannot prevent the presence of large singular sets, even in dimensions $4\leqslant n\leqslant 7$ (see \cite{WhiteLeastAreaDomains}). In particular there even seems to be \emph{no known higher dimensional analogue of Theorem \ref{MN}}. Therefore, to state our theorems for $n\geqslant 4$, we consider both a closely related invariant called $\mathrm{MinA}$, and a higher dimensional width analogue both of which we introduce next. 

\begin{defn}[High Dimensional ``Width'']\label{Wdef}
    Let $n\geqslant 3$, and let \[g=ds^2+f(s)^2g_{rd} \text{\quad with }s\in[0,D]\] be a smooth, rotationally symmetric metric on the $n$-sphere $\mathbb S^n$. The canonical sweepout $\{\Sigma_s\}_{s\in [0, D]}$ by $(n-1)$-spheres is defined analogously to the $n=3$ case, and we define \[W_g:=\max_{s\in [0,D]}\mathrm{Vol}^{n-1}_g(\Sigma_s).\]  
\end{defn} We will state and can prove our main results using this clearly weaker invariant (in particular, we have not \emph{infimized} over any collection of sweepouts) because the largest leaf in the canonical sweepout strongly controls the global geometry of a rotationally symmetric $(\mathbb S^n, g)$ with the appropriate curvature bounds, as we will demonstrate. The other invariant, called $\mathrm{MinA}$, is a familiar quantity appearing in stability problems in the context of the Sormani--Wenger Intrinsic Flat theory (see Section \ref{SectionSWIF}) and is defined as follows: 

\begin{defn}[MinA]
    Let $(M^n,g)$ be a Riemannian $n$-manifold. Then \[\mathrm{MinA}_g:=\inf\left\{\mathrm{Vol}_g^{n-1}(\Sigma): \Sigma\text{ is a closed minimal hypersurface in }M\right\}.\]
\end{defn} Clearly, this is the \emph{strongest} size invariant of the three that we have introduced, in the sense that a lower bound on $\mathrm{MinA}_g$ gives a lower bound on $W_g$ (in all dimensions $n\geq 3$) -- see Theorem \ref{MN} and Lemma \ref{oneminsurf}. This invariant, originally introduced by Sormani in \cite{SormaniSWIFProblems}, serves to control \emph{bubbling} and \emph{spine formation} for sequences of manifolds converging in the Intrinsic Flat sense.

We end this section with a lemma which we will find useful in the proof of our main theorem. While a general rotationally symmetric metric on the sphere may have many minimal hypersurfaces which could be either stable or unstable, such a metric with nearly extremal scalar curvature and MinA bounds cannot.

\begin{lem}\label{oneminsurf}
    Let $n\geqslant 3$. Then there exists an $\varepsilon=\varepsilon(n)>0$ such that if \[g=ds^2+f(s)^2g_{rd} \text{\quad with }s\in[0,D]\] is a smooth, rotationally symmetric metric on the $n$-sphere $\mathbb S^n$ satisfying
    \begin{itemize}
        \item $\mathrm{Scal}_g\geqslant n(n-1)(1-\varepsilon)^2$;
        \item $\mathrm{MinA}_g\geqslant \omega_{n-1}(1-\varepsilon)^{n-1}$, 
    \end{itemize} 
    then the canonical sweepout $\{\Sigma_s\}_{s\in [0, D]}$ of $(\mathbb S^n,g)$ contains exactly one minimal hypersurface, which is necessarily unstable. If we further assume that \begin{itemize}
        \item $\mathrm{Ric}_g\geqslant 0,$
    \end{itemize}
    then the same conclusion holds if we replace $\mathrm{MinA}_g$ with $W_g$. 
\end{lem}

\begin{proof}
Since $f(0)=f(D)=0$, there is some minimal sphere $\Sigma_{s_0}$ where $f(s_0)=\max_{s\in [0,D]} f(s)\geqslant (1-\varepsilon)$. By maximality, $f''(s_0)\leqslant 0$. For the sake of contradiction, suppose that there were another minimal sphere, say $\Sigma_{t_0}$ where without loss of generality $s_0<t_0<D$. We will obtain a contradiction by showing the existence of a stable $\Sigma_{r_0}$ of large area, which is impossible because of the lower scalar curvature bound. Indeed, if $\Sigma_{r_0}$ is any stable minimal sphere in the sweepout, then as is well known the stability inequality and traced Gauss equation yield \[\int_{\Sigma_{r_0}} \mathrm{Scal}_{\Sigma_{r_0}} -|A_{\Sigma_{r_0}}|^2 \geq \int_{\Sigma_{r_0}} \mathrm{Scal}_M.\] Applying the scalar curvature lower bound, we obtain that 

\[\frac{(n-1)(n-2)}{f(r_0)^2}\mathrm{Vol}^{n-1}_g(\Sigma_{r_0})=\int_{\Sigma_{r_0}}\mathrm{Scal}_{\Sigma_{r_0}}\geqslant n(n-1)(1-\varepsilon)^2\mathrm{Vol}^{n-1}_g(\Sigma_{r_0})\] which yields the following upper bound on the radius of $\Sigma_{r_0}$: \[f(r_0)\leqslant \frac{1}{1-\varepsilon}\sqrt{\frac{n-2}{n}}.\] Therefore, if $\mathrm{MinA}_g\geqslant\omega_{n-1}(1-\varepsilon(n))^\frac{n-1}{2}$ where $\varepsilon(n)<1-\left(\frac{n-2}{n}\right)^\frac{n-1}{2}$ is fixed, then $\{\Sigma_s\}_{s\in [0,D]}$ contains no stable minimal surfaces since \[\left(\frac{\mathrm{MinA}_g}{\omega_{n-1}}\right)^\frac{1}{n-1}\leqslant f(r_0).\] 

Thus, for $\mathrm{MinA}_g$ so large it suffices to find a stable minimal $\Sigma_{r_0}$ to obtain a contradiction. If either of $\Sigma_{s_0}$, $\Sigma_{t_0}$ is stable then we are done, and otherwise we have $f'(s_0)=f'(t_0)=0$ with $f''(s_0), f''(t_0)<0$. Then $f$ attains a local minimum at some $r_0$ in the interior of $[s_0, t_0]$ where $f'(r_0)=0$ and $f''(r_0)\geqslant 0$, so that $\Sigma_{r_0}$ is a stable minimal sphere and we may again conclude. 

In the case where $\mathrm{Ric}_g\geqslant 0$, we also have that $f''(x)\leqslant 0$ on $[0,D]$. So $f'(s_0)=0$ and $f''(s_0)\leqslant 0$. In fact, by the proof of Proposition \ref{diameter} (which is independent of this lemma) it follows that $f''(s_0)<0$. Thus, $f'(s)<0$ for every $s\in (s_0, D]$, and there are no further minimal spheres in $\{\Sigma_s\}_{s\in (s_0, D]}$. A similar analysis shows that there are none in $\{\Sigma_s\}_{s\in [0, s_0)}$, so we are done. 
\end{proof}

\subsection{Gromov--Hausdorff Convergence}
Here we will review the Gromov--Hausdorff ($\mathrm{GH}$) distance between two metric spaces. Gromov defined this distance between two metric spaces by generalizing the concept of Hausdorff distance between two subsets of a metric space. We refer the reader to \cite{Gr1} and \cite{BBI} for further details.

The Gromov-Hausdorff distance between two compact metric spaces $(X_1,d_1)$ and $(X_2,d_2)$ can be defined by
\[
d_{\mathrm{GH}}((X_1,d_1),(X_2,d_2))=\inf_{Z} \{\mathrm{d}_{\mathcal H}^Z(\phi_1(X_1),\phi_2(X_2))\}
\]
where the infimum is taken over all complete metric spaces $(Z,d^Z)$ and all distance preserving maps $\phi_i:X_i\to Z$, and where $\mathrm{d}_{\mathcal H}^Z$ denotes the standard Hausdorff distance between two compact subsets of $(Z,d^Z)$: For any compact  $X, Y\subset Z$,
\[\mathrm{d}_{\mathcal H}^Z(X,Y)=\inf\{r>0: X\subset B_r(Y)\text{ and } Y\subset B_r(X)\}\] with $B_r(\cdot)$ denoting the $r$ neighborhood of a subset of $(Z,d^Z)$. We say that a metric spaces $(X_j,d_j)$ converge in the $\mathrm{GH}$-sense to a metric space $(X_\infty,d_\infty)$ if
\[
d_{\mathrm{GH}}((X_j,d_j),(X_\infty,d_\infty)) \to 0.
\] 

We note $\mathrm{GH}$-distance defines a distance between two Riemannian manifolds since we can naturally view a Riemannian manifold as a metric space $(M,d_g)$ where $d_g$ is the induced distance function from $g$.

\subsection{Sormani--Wenger Intrinsic Flat Convergence}\label{SectionSWIF}
In this section, we will review Sormani--Wenger intrinsic flat distance between two integral current spaces. Sormani and Wenger \cite{SW} defined intrinsic flat distance, which generalizes the notion of flat distance for currents in Euclidean space. To do so they used Ambrosio and Kirchheim's \cite{AK} generalization of Federer and Fleming's \cite{FF} integral currents to metric spaces. We refer the reader to \cite{AK} for further details about currents in arbitrary metric spaces and to \cite{SW} for further details about integral current spaces and intrinsic flat distance.

Let $(Z,d^Z)$ be a complete metric space. Denote by $\mathrm{Lip}(Z)$ and $\mathrm{Lip}_b(Z)$ the set of real-valued Lipschitz functions on $Z$ and the set of bounded real-valued Lipschitz functions on $Z$. 
\begin{defn}[\cite{AK}, Definition 3.1]
We say a multilinear functional 
\[
T:\mathrm{Lip}_b(Z)\times [\mathrm{Lip}(Z)]^m\to \bR
\]
on a complete metric space $(Z,d)$ is an $m$-dimensional current if it satisfies the following properties.
\begin{enumerate}[(i)]
    \item Locality: $T(f,\pi_1,\ldots,\pi_m)=0$ if there exists an $i$ such that $\pi_i$ is constant on a neighborhood of $\{f\neq 0\}$.
    \item Continuity: $T$ is continuous with respect to pointwise convergence of $\pi_i$ such that $\mathrm{Lip}(\pi_i)\leq 1$.
    \item Finite mass: there exists a finite Borel measure $\mu$ on $X$ such that 
    \begin{equation} \label{fm}
        |T(f,\pi_1,\ldots,\pi_m)|\leq \prod_{i=1}^m \mathrm{Lip}(\pi_i) \int_Z |f| d\mu
    \end{equation}
    for any $(f,\pi_1,\ldots,\pi_m)$.
 \end{enumerate}
\end{defn}
We call the minimal measure satisfying (\ref{fm}) the mass measure of $T$ and denote it $||T||$. We can now define many concepts related to a current. $\mathbf{M}(T)=||T||(Z)$ is defined to be the mass of $T$ and the canonical set of a $m$-current $T$ on $Z$ is 
\[
\text{set}(T)=\left\{p\in Z \text{ }\Big|\text{ } \liminf_{r\to 0} \frac{||T||(B(p,r))}{r^m}>0\right\}.
\]
The support of $T$ is
\[
\mathrm{spt}\left(T\right):= \mathrm{spt} ||T|| = \left\{z\in Z:||T||\left(B_z\left(r\right)\right)>0\quad \forall r >0\right\}.
\]
Moreover, the closure of $\mathrm{set}\left(T\right)$ is $\mathrm{spt}\left(T\right)$.
The boundary of a current $T$ is defined as $\partial T:\mathrm{Lip}_b(X)\times [\mathrm{Lip}(X)]^{m-1}\to \bR$, where
\[
\partial T(f,\pi_1,\ldots,\pi_{m-1})=T(1,f,\pi_1,\ldots,\pi_{m-1}).
\]
Given a Lipschitz map $\phi:Z\to Z'$, we can pushforward a current $T$ on $Z$ to a current $\phi_\# T$ on $Z'$ by defining
\[
\phi_{\#} T(f,\pi_1,\ldots,\pi_m) = T(f\circ \phi,\pi_1\circ \phi ,\ldots,\pi_m\circ \phi).
\]
A standard example of an $m$-current on $Z$ is given by
\[
\phi_\#\left\llbracket\theta\right\rrbracket(f,\pi_1,\ldots,\pi_m) =\int_A (\theta \circ \phi) (f\circ \phi)d(\pi_1\circ \phi)\wedge \cdots \wedge d(\pi_m\circ \phi),
\]
where $\phi:A\subseteq\bR^m\to Z$ is bi-Lipschitz and $\theta\in L^1(A,\bZ)$.
We say that an $m$-current on $Z$ is integer rectifiable if there is a countable collection of bi-Lipschitz maps $\phi_i: A_i\to X$ where $A_i\subseteq \bR^m$ are precompact Borel measurable with pairwise disjoint images and weight functions $\theta_i\in L^1(A_i,\bZ)$ such that 
\[
T=\sum_{i=1}^\infty \phi_{i\#} \left\llbracket\theta_i\right\rrbracket. 
\]
Moreover, we say an integer rectifiable current whose boundary is also integer rectifiable is an integral current. We denote the space of integral $m$-currents on $Z$ as $\mathbf{I}_m(Z)$. We say that the triple $(X,d,T)$ is an $m$-dimensional integral current space if $(X,d)$ is a metric space, $T\in\mathbf{I}_m(\bar{X})$ where $\bar{X}$ is the metric completion of $X$, and $\text{set}(T)=X$. 
\begin{example}
    Let $(M^n,g)$ be a closed oriented Riemannian manifold. Then there is a naturally associated $n$-dimensional integral current space $\left(M,d_g,\left\llbracket M\right\rrbracket\right)$, where $d_g$ is the distance function induced by the metric $g$ and $\left\llbracket M \right  \rrbracket: \mathrm{Lip}_b\left(M\right)\times \left[\mathrm{Lip}\left(M\right)\right]^{n}\to \bR$ is given by 
    \[
    \left\llbracket M \right  \rrbracket=\sum_{i,j} \psi_{i\#} \left\llbracket \mathbbm{1}_{A_{ij}} \right  \rrbracket
    \]
    where we have chosen a smooth locally finite atlas $\left\{\left(U_i,\psi_i\right)\right\}_{i\in\bN}$ of $M$ consisting of positively oriented biLipschitz charts, $\psi_i: U_i\subseteq \bR^n\to M$ and $A_{ij}$ are precompact Borel sets such hat $\psi_i\left(A_{ij}\right)$ have disjoint images for all $i,j$ and cover $M$ $\scH^n$ almost everywhere. Moreover, we have $\left|\left|\left\llbracket M \right  \rrbracket\right|\right|=d\mathrm{Vol}^n_g$. 
\end{example}
The flat distance between two integral currents $T_1$, $T_2\in\mathbf{I}(Z)$ is
\[
\mathrm{d}^Z_\mathcal{F}(T_1,T_2) = \inf\{\mathbf{M}(U)+\mathbf{M}(V)\mid U\in\mathbf{I}_m(X), V\in\mathbf{I}_{m+1}(X), T_2-T_1=U+\partial V\}.
\]
The intrinsic flat $(\scF)$ distance between two integral current spaces $(X_1,d_1,T_1)$ and $(X_2,d_2,T_2)$ is
\[
\mathrm{d}_\scF((X_1,d_1,T_1),(X_2,d_2,T_2)) = \inf_Z\{\mathrm{d}_{\mathcal F}^Z(\phi_{1\#}T_1,\phi_{2\#}T_2)\},
\]
where the infimum is taken over all complete metric spaces $(Z,d^Z)$ and isometric embeddings $\phi_1:(\bar{X}_1,d_1)\to (Z,d^Z)$ and $\phi_2:(\bar{X}_2,d_2)\to (Z,d^Z)$. We note that if $(X_1,d_1,T_1)$ and $(X_2,d_2,T_2)$ are precompact integral current spaces such that 
\[
\mathrm{d}_\scF((X_1,d_1,T_1),(X_2,d_2,T_2))=0
\]
then there is a current preserving isometry between $(X_1,d_1,T_1)$ and $(X_2,d_2,T_2)$, i.e., there exists an isometry $f:X_1\to X_2$ whose extension $\bar{f}:\bar{X}_1\to \bar{X}_2$ pushes forward the current: $\bar{f}_\#T_1=T_2$. We say a sequence of $(X_j,d_j,T_j)$ precompact integral current spaces converges to $(X_\infty,d_\infty,T_\infty)$ in the $\scF$-sense if
\[
\mathrm{d}_\scF((X_j,d_j,T_j),(X_\infty,d_\infty,T_\infty))\to 0.
\]
If, in addition, $\mathbf{M}(T_i)\to \mathbf{M}(T_\infty)$, then we say $(X_j,d_j,T_j)$ converges to $(X_\infty,d_\infty,T_\infty)$ in the volume preserving intrinsic flat $(\mathcal{VIF})$ sense.

In this paper, our main tool for estimating the Gromov-Hausdorff and Intrinsic-Flat distance between two spaces is the following key result of Lakzian and Sormani in \cite{LS}. In it, two manifolds are found to be close in one of these distances if there are large diffeomorphic subregions situated similarly in each space which are themselves close in the $C^0$ Cheeger--Gromov sense (defined in the following). Specifically, they proved the following:

\begin{thm}[Lakzian--Sormani \emph{cf. \cite{LS}}]\label{LakzianSormani} Suppose that $M_1=(M, g_1)$ and $M_2=(M,g_2)$ are oriented precompact Riemannian manifolds with diffeomorphic subregions $U_i\subset M_i$ and diffeomorphisms $\psi_i:U\to U_i$ such that \[\frac{1}{(1+\varepsilon)^2}\psi_2^*g_2(v,v)\leqslant\psi_1^*g_1(v,v)\leqslant (1+\varepsilon)^2\psi_2^*g_2(v,v)\] for every $v\in TU$ (i.e. $g_1$ and $g_2$ are close in the $C^0$ Cheeger-Gromov sense). We define the following quantities: 
\begin{itemize}
    \item $D_{U_i}=\sup\{\mathrm{diam}_{M_i}(W): \text{W is a connected component of $U_i$}\}$;
    \item $a>\pi^{-1}\arccos\left((1+\varepsilon)^{-1}\right)\max\{D_{U_1}, D_{U_2}\}$;
    \item $\lambda=\sup_{x,y\in U} |d_{M_1}(\psi_1(x),\psi_1(y))-d_{M_2}(\psi_2(x),\psi_2(y))|$;
    \item $h=\sqrt{
    \lambda(\max\{D_{U_1}, D_{U_2}\}+\lambda/4)
    }$;
    \item $\overline{h}=\max\{h, \sqrt{\varepsilon^2+2\varepsilon}D_{U_1},\sqrt{\varepsilon^2+2\varepsilon}D_{U_2}\}$.
\end{itemize} Then the Gromov--Hausdorff distance between the metric completions of the $M_i$ is bounded: \[\mathrm{d}_{GH}(\overline{M_1}, \overline{M_2})\leqslant a+2\overline h+\max\left\{\mathrm{d}_H^{M_1}(U_1, M_1), \mathrm{d}_H^{M_2}(U_2, M_2)\right\}.\] Similarly, the intrinsic flat distance between the settled completions\footnote{i.e. the set of all points in the metric completion of $M_i$ with positive lower densitiy with respect to $d\mathrm{Vol}_g$.} of the $M_i$ is bounded: 
\begin{align*}
    \mathrm{d}_{\scF}(M_1', M_2')&\leqslant \left(a+2\overline h\right)\left(\mathrm{Vol}^n_{g_1}(U_1)+\mathrm{Vol}^n_{g_2}(U_2)+\mathrm{Vol}^{n-1}_{g_1}(\partial U_1)+\mathrm{Vol}^{n-1}_{g_2}(\partial U_2)\right)\\ 
    &+\mathrm{Vol}^{n}_{g_1}(M_1\setminus U_1)+\mathrm{Vol}^{n}_{g_2}(M_2\setminus U_2).
\end{align*}
\end{thm}

\section{A Width Rigidity Theorem in all Dimensions \'a la Marques-Neves}\label{sec: width rig}

In this section, we prove a version of Marques-Neves' Theorem \ref{MN} for rotationally symmetric manifolds in \emph{all dimensions} $n\geqslant 3$. We emphasize that Theorem \ref{MN} is manifestly a \emph{three-dimensional} result, on account of the fact that it concerns the Simon--Smith width which has no known higher dimensional analogue in generality (cf. Section \ref{minmax}). Rotational symmetry, however, allows us to introduce the related, weaker invariant $W_g$ of Definition \ref{Wdef}, which is clearly defined for every $n\geqslant 3$. With this extended definition of ``width'', we can state and prove a rotationally symmetric version of Theorem \ref{MN} which is valid in all dimensions $n\geqslant 3$, and moreover requires \emph{no} assumption on the Ricci curvature. The following is a restatement of \Cref{t: rigidity}:
\begin{thm}[Rigidity of the Width in Rotational Symmetry]
Let $n\geq3$ and $g$ be a rotationally symmetric Riemannian metric on the n-sphere, $\bS^n$ such that
\begin{itemize}
    \item $\mathrm{Scal}_g\geqslant n(n-1)$;
    \item $W_g\geqslant\omega_{n-1}$,
\end{itemize} where $W_g$ is the Simon--Smith width when $n=3$, and is as in Definition \ref{Wdef} for $n\geqslant 4$. Then $(\mathbb S^n, g)$ is isometric to $(\mathbb S^n, g_{rd})$.
\end{thm}

\begin{proof}
Since $(\mathbb S^n,g)$ is rotationally symmetric, we can write it in the isometric form 
\[
\left([0,1]\times \mathbb S^{n-1}, ds^2+f(s)^2g_{rd}\right)
\]
where $f:[0,1]\to \mathbb R$ is a non-negative smooth function such that $f(0)=f(1)=0$ and $f'(0)=-f'(1)=1$. Therefore, we know that $f(s)$ obtains an interior maximum at some point $m^*\in(0,1)$, and by the definition of $W_g$ for each $n\geqslant 3$, we have that 
\[
\omega_{n-1}\leqslant W_g \leqslant \omega_{n-1}f(m^*)^{n-1}.
\]
Now, consider the restricted metric $\overline g=ds^2+f(s)^2g_{rd}$ on the hemisphere $\mathbb S^n_+=[0,m^*]\times \mathbb S^{n-1}$, and note that $\mathrm{Scal}_{\overline{g}} \geqslant n(n-1)$. Furthermore, $\partial \mathbb S^n_+$ is totally geodesic with induced metric $f(m^*)^2g_{rd}$. By the above observation, $f(m^*)\geqslant 1$ so we also obtain $\overline{g}\geqslant g_{rd}$ along $\partial\mathbb S^n_+$. 

Next, note that by virtue of its rotational symmetry, $\overline g$ is conformal to the Euclidean metric on the Euclidean $n$-ball of radius $1$, $B^n(1)$. In particular, $(\mathbb S^n_+, \overline g)$ is isometric to $\left(B^n(1), \tilde g:= u^\frac{4}{n-2}g_{euc}\right)$ for some smooth function $u:B^n(1)\to\mathbb R^{>0}$. As is well known, the conformal factor $u$ obeys the Yamabe equation
\[
\mathrm{Scal}_g = \mathrm{Scal}_{\tilde g} = u^{-\frac{n+2}{n-2}}\left(-\frac{4(n-1)}{n-2} \Delta u\right)
\] 
where $\Delta$ is the standard Euclidean Laplacian. Moreover, the fact that $\tilde g\geqslant g_{rd}$ along the boundary implies that $u\vert_{\partial B_1^n}\geqslant 1.$ Therefore, since $\mathrm{Scal}_g\geqslant n(n-1)$, $u$ satisfies the PDE problem \begin{align*} \label{scalpde}
    \begin{cases}
        - \Delta u \geq\frac{n(n-2)}{4} u^{\frac{n+2}{n-2}} \qquad &\text{in }  B^n(1)\\
        u\geqslant 1 \qquad &\text{on }  \partial B^n(1),
    \end{cases}
\end{align*} and by Claim 3.5 of Hang--Wang in \cite{HW}, the only possible solution is \[u(x)=\left(\frac{2}{1+|x|^2}\right)^{\frac{n-2}{2}}.\] This is exactly the conformal factor of the standard round metric $g_{rd}$ on $\mathbb S^n_+$ with respect to $g_{euc}$ on $B^n(1)$, and applying the same argument to the other hemisphere of $(\mathbb S^n, g)$ thus establishes the Theorem. 
\end{proof}

\section{A Priori Diameter Estimate}\label{sec: a priori}

In this section we prove an explicit a priori upper bound for the diameter of a rotationally symmetric $n$-sphere with positive scalar curvature, non-negative Ricci curvature, and width bounded below by a dimensional constant. 

\begin{prop}\label{diameter}
    Let $n\geqslant 3$, $\Lambda>0$, $w_0>\sqrt{\frac{(n-1)(n-2)}{\Lambda}}\coloneqq 1/\sqrt{\tilde \Lambda}$. Then there exists a $D_0=D_0(n, \Lambda, w_0)<\infty$ such that if \[g=ds^2+f(s)^2g_{rd} \text{\quad with }s\in[0,D]\] is a smooth, rotationally symmetric metric on the $n$-hemisphere $\mathbb S^n_+$ satisfying
    \begin{itemize}
        \item $\mathrm{Scal}_g\geqslant \Lambda>0$;
        \item $\mathrm{Ric}_g\geqslant 0$;
        \item $f(0)\geqslant w_0>0$, $f'(0)=0$, and $f(D)=0$, 
    \end{itemize} 
    then $D\leqslant D_0$. In particular, if $(\mathbb S^n, g)$ is rotationally symmetric and satisfies the above curvature conditions with $W_g\geqslant\omega_{n-1}w_0^{n-1}$, then $\mathrm{diam}_g\leqslant 2D_0$. 
    
    Moreover, one may take the explicit value
    \[
    D_0=\frac{2nw_0}{n-2}\cdot\frac{1}{\tilde{\Lambda}w_0^2-\log\left(\tilde\Lambda w_0^2\right)-1}.
    \]
\end{prop} 

\begin{rem}
    The threshold value of $1/\sqrt{\tilde\Lambda}$ for the lower bound on $w_0$ is sharp, as can be seen by considering long and thin ellipsoids opening up to the constant scalar curvature $\Lambda$ cylinder. We also recall, in dimension $n=3$, that $8\pi/\Lambda$ is the upper threshold of area for a \emph{stable}, embedded, closed, orientable minimal surface (necessarily a two-sphere) in an $(M^3, g)$ with $\mathrm{Scal}_g\geqslant \Lambda>0$ (see eg. \cite{MN} Proposition A.1(i)). The explicit value of $D_0$ is certainly not sharp, however.  
\end{rem}

\begin{rem}
    It is easy to see that each of these curvature assumptions is necessary for a universal bound on diameter. Without a positive scalar curvature lower bound, the sole assumption of non-negative Ricci curvature cannot force the hemisphere to ``close up'' in a controlled manner, as one can see by considering capped-off cylinders of unbounded length. Without the assumption of non-negative Ricci curvature, unduloid-like strings of spheres joined by thick necks with arbitrarily small negative Ricci lower bounds exist. 
\end{rem}

\begin{proof}
    Let $f:[0,D]\to [0,1]$ be such a warping function for a metric $g$ on $\mathbb S^n_+$, and recall that the curvature conditions in the hypotheses enforce the following differential inequalities: \begin{equation*}
    f''(x)\leqslant 0 \qquad\text{and}\qquad \Lambda\leqslant (n-1)(n-2)\left(\frac{1-f'(x)^2}{f(x)^2}-\frac{2}{(n-2)}\frac{f''(x)}{f(x)}\right).
 \end{equation*} Together with $f(0)\geqslant w_0$, the first inequality tells us that on $[0,D]$\[f(x)\geqslant \left(1-\frac{x}{D}\right)f(0)\geqslant\left(1-\frac{x}{D}\right)w_0.\] Rearranging the second inequality, plugging in the above lower bound on $f(x)$, and using Lemma \ref{ScalarLipschitz} tells us that on $[0,D]$ \[f''(x)\leqslant p(x)\coloneqq\frac{n-2}{2}\left(\frac{1}{\left(1-\frac{x}{D}\right)w_0}-\tilde\Lambda w_0\left(1-\frac{x}{D}\right)\right).\] Since $w_0>1/\sqrt{\tilde\Lambda}$, we have that $p(0)<0$ and $p(x)$ is increasing. Moreover,
 \[
 \delta\coloneqq \sup\left\{0< x\leqslant D : p(x)<0 \text{ on } [0, x)\right\}>0,
 \]
 which is explicitly calculated to be 
 \[
 \delta=D\left(1-\frac{1}{w_0\sqrt{\tilde\Lambda}}\right).
 \]

 Now, recalling that $f'(0)=0$, $f'(D)=-1$, and $f''(x)\leqslant 0$, we compute from the above that 
 \[
 -1=\int_0^D f''(s)ds\leqslant \int_0^\delta p(s)ds=\frac{n-2}{2nw_0}\left(1-\tilde\Lambda w_0^2+\log\left(\tilde\Lambda w_0^2\right)\right)D<0
 \] 
 from which the claim follows. 
\end{proof}

\section{Stability of the Width: Proof of the Main Theorems}\label{sec: stab}

In this section we prove our main results, which will follow immediately from the following theorem, Lemma \ref{oneminsurf}, and Proposition \ref{diameter}:

\begin{thm}\label{main}
    Fix $n\geqslant 3$, $D>0$, and $\delta>0$. There exists an $\varepsilon=\varepsilon(n, \delta, D)>0$ such that if $g$ is a rotationally symmetric metric on the $n$-sphere $\mathbb S^n$ satisfying 
    \begin{itemize}
        \item $\mathrm{diam}_g(\mathbb S^n)\leqslant D$;
        \item $\mathrm{Scal}_g\geqslant n(n-1)(1-\varepsilon)^2$;
        \item $\mathrm{MinA}_g\geqslant \omega_{n-1}(1-\varepsilon)^{n-1},$ 
    \end{itemize} then $\mathrm{d}_{\mathcal{VIF}}((\mathbb S^n, g), (\mathbb S^n, g_{rd}))\leqslant\delta$.
\end{thm}

Without any further assumptions, and notably without any further curvature bounds, we obtain the following more-or-less equivalent rephrasing of Theorem \ref{main} in terms of the Gromov--Hausdorff distance. In this case, the potential formation of spines at the poles of the spheres is not controlled by the Gromov--Hausdorff topology, so we must excise this potential ``bad set'' to get Gromov--Hausdorff convergence (this technique has been used to great effect in stability probems before--see eg. \cites{D,DS,HirshZhangLlaurllStab, HI}). The following is a restatement of \Cref{t: mainDSHI}:

\begin{thm}\label{GHSurgery}
     Fix $n\geqslant 3$, $D>0$, and $\delta>0$. There exists an $\varepsilon=\varepsilon(n, \delta, D)>0$ such that if $g$ is a rotationally symmetric metric on the $n$-sphere $\mathbb S^n$ satisfying 
    \begin{itemize}
        \item $\mathrm{diam}_g(\mathbb S^n)\leqslant D$;
        \item $\mathrm{Scal}_g\geqslant n(n-1)(1-\varepsilon)^2$;
        \item $\mathrm{MinA}_g\geqslant \omega_{n-1}(1-\varepsilon)^{n-1},$ 
    \end{itemize} then there exists a smooth domain $Z\subset\mathbb S^n$ with at most two connected components satisfying \[\mathrm{Vol}_g^n(Z)+\mathrm{Vol}_g^{n-1}(\partial Z)\leqslant \delta,\] so that $\mathrm{d}_{GH}((\mathbb S^n\setminus Z, g), (\mathbb S^n, g_{rd}))\leqslant\delta$.
\end{thm}

If we further impose the condition that $g$ has non-negative Ricci curvature, then we can obtain Gromov--Hausdorff stability without the assumed diameter bound and without the surgical removal of the bad set using the proof of Theorem \ref{main} and Proposition \ref{diameter}. In this situation, we can also phrase the $\mathrm{MinA}_g$ condition in terms of $W_g$ instead, by the last conclusion of Lemma \ref{oneminsurf}. Indeed, all that is important in the proof of Theorem \ref{main} is that no leaf of the sweepout $\Sigma_s$ is minimal if $s\neq 0$. Thus we obtain the following (a restatement of Theorem \ref{t: mainRicci}), which is another stabilized version of Marques-Neves' Theorem \ref{MN} in general $n$-dimensions under the assumption of rotational symmetry: 

\begin{thm}\label{GHRicci}
    Fix $n\geqslant 3$ and $\delta>0$. There exists an $\varepsilon=\varepsilon(n, \delta)>0$ such that if $g$ is a rotationally symmetric metric on the $n$-sphere $\mathbb S^n$ satisfying 
    \begin{itemize}
        \item $\mathrm{Scal}_g\geqslant n(n-1)(1-\varepsilon)^2$;
        \item $\mathrm{Ric}_g\geqslant 0$;
        \item $W_g\geqslant \omega_{n-1}(1-\varepsilon)^{n-1},$ 
    \end{itemize} then $\mathrm{d}_{GH}((\mathbb S^n, g), (\mathbb S^n, g_{rd}))\leqslant\delta$.
\end{thm}

\begin{proof}[Proof of Theorem \ref{main}]

We proceed by way of contradiction, supposing that the result were false and obtaining some $\delta_0>0$ and a sequence of smooth counterexample metrics on $\mathbb S^n$, written in coordinates as \[g_k=ds^2+f_k(s)^2g_{rd}\text{\quad on\quad} [0,S_k]\times\mathbb S^{n-1},\]for $k=1, 2, \ldots$ satisfying \begin{itemize}
    \item $S_k=\mathrm{diam}_{g_k}(\mathbb S^n)\leqslant D$;
    \item $\mathrm{Scal}_{g_k}\geqslant n(n-1)(1-k^{-1})^2$;
    \item $\mathrm{MinA}_{g_k}\geqslant \omega_{n-1}(1-k^{-1})^{n-1}$; 
    \item $\mathrm{d}_{\mathcal{VIF}}((\mathbb S^n, g_k), (\mathbb S^n, g_{rd}))\geqslant\delta_0$ for all $k\geqslant 1$.
\end{itemize} To begin, we utilize  \Cref{oneminsurf} for all large enough $k\geqslant 1$ to identify the sole minimal hypersurface $\Sigma^{(k)}_{\tilde S_k}=\{\tilde S_k\}\times\mathbb S^{n-1}$ in the canonical sweepout of $(\mathbb S^n, g_k)$, which divides $(\mathbb S^n, g_k)$ into two connected hemispheres which we denote by $(\mathbb S^n_\pm, g_k)$. See Figure \ref{fig:rotsymsphere}. For the sake of readability, in the following sequence of lemmas we will only explicitly work on $(\mathbb S^n_-, g_k)$ since the situation with $(\mathbb S^n_+, g_k)$ is handled nearly identically. As in \Cref{overview}, we introduce $(r,\theta)$ coordinates on the hemisphere so that 
 \begin{align*}
     \left(\mathbb S^n_-, g_k\right)&=_{isom}\left(\left[0,\tilde S_k\right)\times\mathbb S^{n-1}, \quad g_k=ds^2+f_k(s)^2g_{rd}\right)\\ &=_{isom}\left(\left[0,R_k\right)\times\mathbb S^{n-1}, \quad g_k=\frac{1}{V_{k,-}(r)}dr^2+r^2g_{rd}\right).
 \end{align*} 
 \begin{figure}
     \centering
     \includegraphics[width=0.4\linewidth]{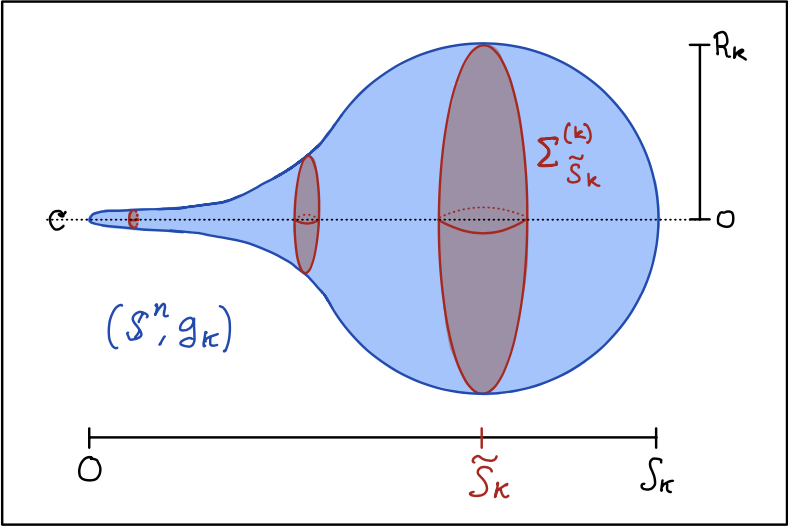}
     \caption{$(\mathbb S^n, g_k)$ with the minimal sphere $\Sigma^{(k)}_{\tilde S_k}$.}
     \label{fig:rotsymsphere}
 \end{figure}
 
 We also recall that this means $r$ is related to $s$ via the formula $r=f_k(s)$, and that on $[0, R_k)$ \[V_{k,-}(r)=\left(f_k'\left(\left(f_k\vert_{\left[0,\tilde S_k\right)}\right)^{-1}(r)\right)\right)^2.\] The following simple lemma bounds $R_k$ and $\tilde S_k$ (recall that by definition, $R_k=f_k(\tilde S_k)$):

\begin{lem}\label{Rk}
For every $k\geqslant 1$, we have that \[1-k^{-1}\leqslant R_k\leqslant \frac{1}{1-k^{-1}} \quad\text{and}\quad (1-k^{-1})\leqslant \tilde S_k\leqslant S_k.\]   
\end{lem}

\begin{proof}
    Recall that \[\mathrm{Scal}_{g_k}=\frac{n-1}{r^2}\left((n-2)(1-V(r))-rV'(r)\right)\geqslant n(n-1)(1-k^{-1})^2\] which after suitable rearrangement becomes \[\left(r^{n-2}(1-V(r))\right)'\geqslant nr^{n-1}(1-k^{-1})^2.\] If we integrate both sides from $0$ to $R_k$ using $V_k(R_k)=0$, we obtain the upper bound on $R_k$. The lower bound on $R_k$ simply follows from the $\mathrm{MinA}_{g_k}$ lower bound. By pairing the $\mathrm{MinA}_{g_k}$ lower bound with Lemma \ref{ScalarLipschitz}, we also obtain the lower bound on $\tilde S_k$. 
\end{proof}

For every $k\geqslant 1$ we extend $V_{k,-}(r)$ from $[0, R_k]$ to $[0,1]$ constantly by $0$, and we extend $f_k(s)$ from $[0, \tilde S_k]$ to $[0, D]$ constantly by $R_k$. These extensions are smooth everywhere on their now fixed domains of definition, except at the points $R_k$ and $\tilde S_k$, respectively, where they are continuous. By Lemma \ref{Rk}, we can assume that up to a subsequence, $\tilde S_k\nearrow \tilde S\in [1,D]$. Using the scalar curvature and MinA bounds, we may now prove our two most fundamental estimates on our metric tensors, which give pointwise convergence to the round sphere. The first estimate will eventually allow us to prove volume convergence (\Cref{volumeconv}) and together with the second estimate we will be able to establish all of the other estimates needed (\Cref{C0CheegerGromov}, \Cref{DU}, \Cref{distortion}) to apply Lakzian--Sormani's \Cref{LakzianSormani} to obtain the desired contradiction. 

\begin{lem}[Fundamental Metric Estimates (I)]\label{vestimate} 
Fix $0<\eta<1$. For every $k\geqslant 1$ large enough, we have the following uniform estimate for $r\in [\eta,1]$: \[\left|V_{k,-}(r)-V_{rd}(r)\right|\leqslant \Psi(k^{-1}:\eta)\] where we recall that $V_{rd}(r)=1-r^2$. Therefore if $0<\rho<<1$, we have that \[\left|\frac{V_{k,-}(r)}{V_{rd}(r)}-1\right|\leqslant \Psi(k^{-1}:\eta, \rho)\quad\text{on}\quad [\eta, 1-\rho].\]
\end{lem}

\begin{proof}
Recalling Section \ref{overview}, via the scalar curvature lower bounds we have the following ordinary differential inequality for $V_{k,-}$ on $[0,R_k]$: \[n(n-1)(1-k^{-1})^2\leqslant \mathrm{Scal}_{g_k}=\frac{n-1}{r^2}\left((n-2)(1-V_{k,-}(r))-rV_{k,-}'(r)\right),\] or in other words \[V_{k,-}'(r)+\frac{n-2}{r}V_{k,-}(r)\leqslant \frac{n-2}{r}- n(1-k^{-1})^2r.\] After multiplying this last line by the integrating factor $r^{n-2}$ and integrating from $0$ to $r\in (0,R_k]$, we obtain the upper bound \[V_{k,-}(r)\leqslant 1-(1-k^{-1})^2r^2\text{\quad on\quad} [0,R_k].\] If we instead integrate from $r\in [0,R_k)$ to $R_k$ and recall that $V_{k,-}(R_k)=0$ by minimality, we obtain the lower bound \[(1-(1-k^{-1})^2r^2)+\frac{(1-k^{-1})^2R_k^n-R_k^{n-2}}{r^{n-2}}\leqslant V_{k,-}(r)\text{\quad on\quad} (0,R_k].\] Notice that by Lemma \ref{Rk} the second term on the left hand side could degenerate to $-\infty$ as $r\searrow 0$, but that on $[\eta, R_k]$ it is always bounded and decays to $0$ uniformly as $k\to\infty$. Therefore, recalling the definition of $V_{rd}(r)=1-r^2$ on $[0,1]$, we can easily wrap these estimates into the form \[\left|V_{k,-}(r)-V_{rd}(r)\right|\leqslant \Psi(k^{-1}:\eta)\text{\quad on\quad} [\eta, 1].\] The last estimate in the lemma follows from the first, if we pair it with the bounds $1\geqslant V_{rd}(r)\geqslant 1-(1-\rho)^2$ on $[\eta, 1-\rho]$. 
\end{proof}

\begin{lem}[Fundamental Metric Estimates (II)]\label{festimate} 
For every $k\geqslant 1$ large enough, we have the following uniform estimate for $s\in \left[\tilde S-1,\tilde S\right]$ (recall that $\tilde S=\lim_{k\to\infty}\tilde S_k$):\[\left\vert\frac{f_k(s)}{f_{rd}\left(s+\left(\frac{\pi}{2}-\tilde S\right)\right)}-1\right\vert\leqslant\Psi(k^{-1})\] where we recall that $f_{rd}(s)=\sin(s)$. Therefore, for all large $k$, $f_k(\tilde S_k-1)\leqslant \frac{3}{4}$ so that the two coordinate systems \[(r,\theta)\in U_{k,\eta, -}:=\left(\eta, R_k-\frac{1}{100}\right)\times\mathbb S^{n-1}\quad\text{and}\quad (s,\theta)\in E_{k,-}:=\left(\tilde S_k-\frac{99}{100}, \tilde S_k\right]\times\mathbb S^{n-1}\] cover all of $(\mathbb S^n_-, g_k)$ except the ``small'' region $[0,\eta]\times\mathbb S^{n-1}$ in the $(r,\theta)$ coordinates, and where both of these coordinate charts enjoy the estimates of this and the previous lemma. 
\end{lem}

\begin{proof}
We first observe that by our upper diameter bound $D$ and the a-priori Lipschitz bound of Lemma \ref{ScalarLipschitz}, the Arzel\'a-Ascoli Theorem guarantees that a subsequence of the $f_k$ converges uniformly on $[0,\tilde S]$ to a non-negative $1$-Lipschitz function $f_\infty$ which clearly satisfies $f_\infty(0)=0$ and $f_\infty(\tilde S)=1$ (by Lemma \ref{Rk}). In fact, we claim that $f_\infty$ satisfies the following partial boundary value problem, from which the result follows easily: \[\begin{cases}
    f_\infty'(s)^2+f_\infty(s)^2=1 & \text{on}\quad (\tilde S-1,\tilde S)\\
    f_\infty(\tilde S)=1. &
\end{cases}\]To establish this, recall that $V_{k,-}(f_k(s))=(f_k'(s))^2$. Fix an arbitrary $0<\eta<1$. For any $s\in (\tilde S-\eta, \tilde S)$, for all large enough $k$ we also have $s\in (\tilde S_k-\eta, \tilde S_k)$ where $f_k$ is smooth. By Lemma \ref{ScalarLipschitz}, we can estimate that \[f_k(s)=R_k-\int_{s}^{\tilde S_k} f'_k(\xi)d\xi\geqslant (1-k^{-1})-(\tilde S_k-s)\geqslant 1-\eta-k^{-1}\geqslant\frac{1-\eta}{2}>0\] for all large enough $k$ (depending on $\eta$). By the uniform convergence of the $V_{k,-}$ on $[(1-\eta)/2, 1]$ just proven in Lemma \ref{vestimate}, we therefore see that we also have uniform convergence of the $f_k'(s)=\left(V_{k,-}(f_k(s))\right)^{1/2}$ on $[\tilde S-\eta, \tilde S-\rho]$ for any $0<\rho<<1$. To wit,
\[
|f_k'(s)^2-V_{k,-}(f_k(s))|\leqslant\Psi(k^{-1}: \eta, \rho)\quad\text{on}\quad [\tilde S-\eta, \tilde S-\rho]
\]
and so by Lemma \ref{vestimate} and the fact that $f_k\to f_\infty$ uniformly on $[0,\tilde S]$, 
\[
|f_k'(s)^2+f_\infty(s)^2-1|\leqslant\Psi(k^{-1}: \eta, \rho)\quad\text{on}\quad [\tilde S-\eta, \tilde S-\rho].
\] 
Thus, by sending $k\to\infty$ we obtain that $f_\infty$ is differentiable and 
\[ f_\infty'(s)^2+f_\infty(s)^2=1\quad\text{on}\quad [\tilde S-\eta, \tilde S-\rho],
\]
from which smoothness follows. Sending $\eta\nearrow 1$ and $\rho\searrow 0$ establishes the claim, from which we see that 
\[
f_\infty(s)=\sin\left(s+\left(\frac{\pi}{2}-\tilde S\right)\right)\quad\text{on}\quad [\tilde S-1,\tilde S].
\] 
The conclusion of the lemma now follows readily. 
\end{proof}

\begin{rem}
Our arguments require us to use both of these coordinate systems in order to cover enough of the manifold to garner global convergence. Notice that Lemma \ref{festimate} addresses a region of definite size around the largest sphere in the canonical sweepout of each $(\mathbb S^n, g_k)$, and tells us that inside this region we asymptotically see the geometry of the round sphere. However, the possibility of spine formation away from the minimal sphere causes the estimates in these coordinate charts to break down. Nevertheless, Lemma \ref{vestimate} and the uniform diameter bound give us enough control on the rest of the manifold to make up for this. See Figure \ref{fig:coords}.
\end{rem}

\begin{figure}
    \centering
\includegraphics[width=0.85\linewidth]{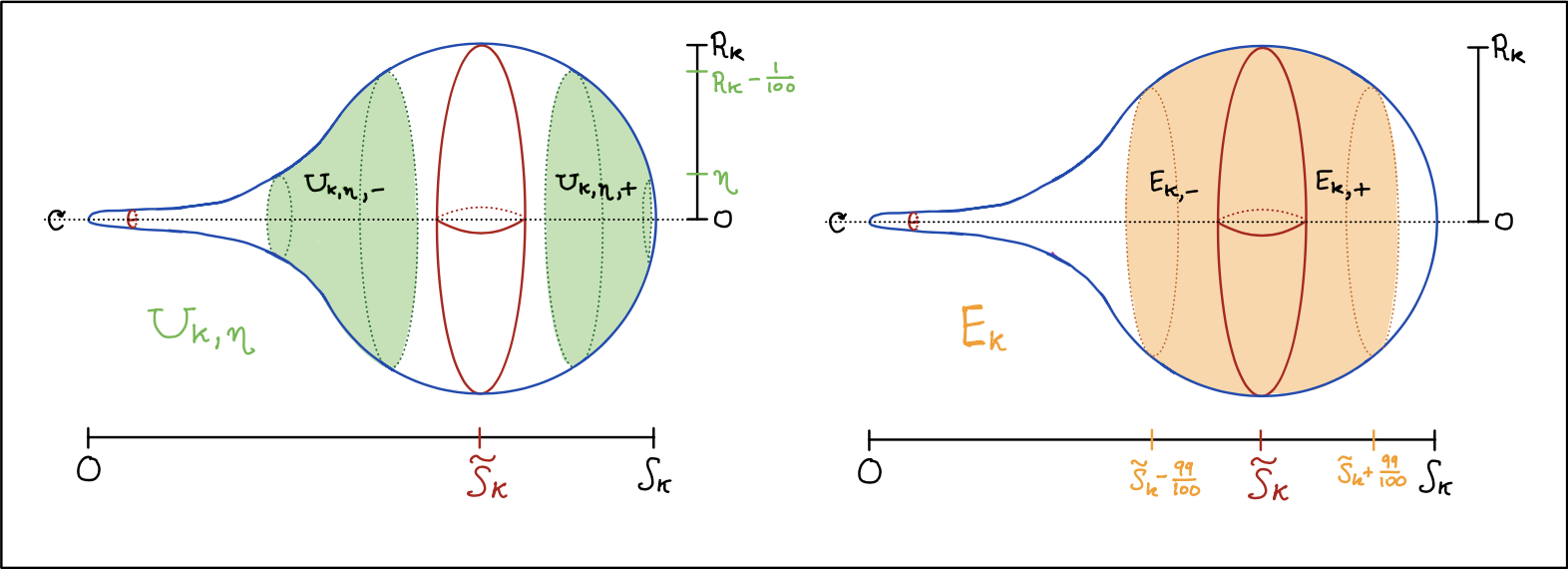}
    \caption{The charts $U_{k,\eta}$ and ${E_k}$.}
    \label{fig:coords}
\end{figure}

In the following sequence of lemmas, we estimate the various quantities appearing in the Lakzian--Sormani estimates of Theorem \ref{LakzianSormani} using Lemmas \ref{vestimate} and \ref{festimate}. We also return to global setting on all of $\mathbb S^n$, having similarly carried out the analogous estimates on the hemispheres $(\mathbb S_+^n, g_k)$, which we cast in the following useful coordinate parametrizations:  \begin{align*} \left(\mathbb S^n_+, g_k\right)&=_{isom}\left(\left(\tilde S_k, S_k\right]\times\mathbb S^{n-1}, \quad g_k=ds^2+f_k(s)^2g_{rd}\right)\\ &=_{isom}\left(\left(2-R_k,2\right]\times\mathbb S^{n-1}, \quad g_k=\frac{1}{V_{k,+}(2-r)}dr^2+(2-r)^2g_{rd}\right)\end{align*} where $2-r=f_k(S_k-s)$. Using the notation established in Lemma \ref{festimate}, we thus work in the following charts for $(\mathbb S^n, g_k)$ (see again Figure \ref{fig:coords}):
\begin{align*}
    &(r,\theta)\in U_{k,\eta}:=U_{k,\eta,-}\sqcup U_{k,\eta, +}=\left(\left(\eta, R_k-\frac{1}{100}\right)\cup \left(2-R_k+\frac{1}{100}, 2-\eta\right)\right)\times\mathbb S^{n-1}\\
    &(s,\theta)\in E_k:=E_{k,-}\cup  E_{k,+}=\left(\tilde S_k-\frac{99}{100}, \tilde S_k+\frac{99}{100}\right)\times\mathbb S^{n-1},
\end{align*} 
where the restriction of $g_k$ (defined in $(r,\theta)$ coordinates for $r\in [0, 2R_k]\setminus\{R_k\}$) has the following coordinate expressions: \begin{align*}
    g_k\vert_{p}=\begin{cases}
         \frac{1}{V_{k,-}(r)}dr^2+r^2g_{rd} & \text{if } p=(r,\theta)\in U_{k, \eta, -}\subset U_{k,\eta}\\
        ds^2+f_k(s)^2g_{rd} & \text{if } p=(s,\theta)\in E_k \\
        \frac{1}{V_{k,+}(2-r)}dr^2+(2-r)^2g_{rd} & \text{if } p=(r,\theta)\in U_{k, \eta, +}\subset U_{k,\eta}.\\
    \end{cases}
\end{align*} We also define the fixed charts  
\begin{align*}
    &(r,\theta)\in U_\eta:=U_{\eta,-}\sqcup U_{\eta, +}=\left(\left(\eta, \frac{9}{10}\right)\cup \left(\frac{11}{10}, 2-\eta\right)\right)\times\mathbb S^{n-1}\\
    &(s,\theta)\in E:=E_-\cup  E_+=\left(\tilde S-\frac{9}{10}, \tilde S+\frac{9}{10}\right)\times\mathbb S^{n-1},
\end{align*}  where the restriction of $g_{rd}$ (defined in $(r,\theta)$ coordinates for $r\in [0,2]\setminus\{1\}$) has the following coordinate expressions: \begin{align*}
    g_{rd}\vert_{p}=\begin{cases}
         \frac{1}{1-r^2}dr^2+r^2g_{rd} & \text{if } p=(r,\theta)\in U_{\eta, -}\subset U_{\eta}\\
        ds^2+\sin\left(s+\left(\frac{\pi}{2}-\tilde S\right)\right)^2g_{rd} & \text{if } p=(s,\theta)\in E \\
        \frac{1}{1-(2-r)^2}dr^2+(2-r)^2g_{rd} & \text{if } p=(r,\theta)\in U_{\eta, +}\subset U_{\eta}.\\
    \end{cases}
\end{align*} Finally, we let \[\Omega_\eta\coloneqq (\eta, 2-\eta)\times\mathbb S^{n-1}\] denote a spherical band. For all $k\geqslant 1$ so large that $|R_k-1|, |\tilde S_k-\tilde S|<\min\{\frac{1}{1000},\frac{\eta}{2}\}$, Lemma \ref{festimate} guarantees that the charts $(U_{k,\eta}, g_k)$ and $(E_k, g_k)$ together cover $\Omega_\eta$ with the pullback of $g_k$ by the natural inclusion induced by the $(r,\theta)$ coordinates, which we write as \[\Omega_{k,\eta}\coloneqq(\Omega_\eta, g_k) \hookrightarrow (\mathbb S^n, g_k).\] Corresponding to this is, of course, the following subregion of the round sphere: \[\Omega_{rd, \eta}:=(\Omega_\eta, g_{rd})\hookrightarrow (\mathbb S^n, g_{rd}).\] Lemmas \ref{vestimate} and \ref{festimate} then allow us to uniformly compare the components of $g_k$ and $g_{rd}$ on $\Omega_\eta$ to easily obtain the first estimate needed for Lakzian--Sormani's Theorem \ref{LakzianSormani}:

\begin{lem}[$C^0$ Cheeger--Gromov Convergence]\label{C0CheegerGromov}
    For all $k\geqslant 1$ large enough, \[1-\Psi(k^{-1}:\eta)\leqslant \frac{g_k(v,v)}{g_{rd}(v,v)}\leqslant 1+\Psi(k^{-1}:\eta)\]  for every $v\in T\Omega_\eta$ (where we have omitted the pullback maps to $\Omega_\eta$ from the notation for readability). \end{lem}

\begin{proof}
    Fix $v_p\in T_p\Omega_\eta$. If $p=(r,\theta)\in U_\eta$, then for all large enough $k\geqslant 1$ and Lemma \ref{vestimate} \begin{align*}
        \frac{g_k(v,v)}{g_{rd}(v,v)}&=\frac{\frac{1}{V_k(r)}dr^2(v,v)+r^2g_{rd}(v,v)}{\frac{1}{\sqrt{1-r^2}}dr^2(v,v)+r^2g_{rd}(v,v)}\\
        &\leqslant\frac{(1+\Psi(k^{-1}:\eta))\left(\frac{1}{\sqrt{1-r^2}}dr^2(v,v)+r^2g_{rd}(v,v)\right)}{\frac{1}{\sqrt{1-r^2}}dr^2(v,v)+r^2g_{rd}(v,v)}\\
        &=1+\Psi(k^{-1}:\eta)
    \end{align*} and similarly for the other inequality. Likewise, if it happens that $p=(s,\theta)\in E$, then for every $k\geqslant 1$ large enough and Lemma \ref{festimate}
    \begin{align*}
        \frac{g_k(v,v)}{g_{rd}(v,v)}&=\frac{ds^2(v,v)+f_k(s)^2g_{rd}(v,v)}{ds^2(v,v)+\sin\left(s+\left(\frac{\pi}{2}-\tilde S\right)\right)^2g_{rd}(v,v)}\\
        &\leqslant\frac{(1+\Psi(k^{-1}:\eta))\left(ds^2(v,v)+\sin\left(s+\left(\frac{\pi}{2}-\tilde S\right)\right)^2g_{rd}(v,v)\right)}{ds^2(v,v)+\sin\left(s+\left(\frac{\pi}{2}-\tilde S\right)\right)^2g_{rd}(v,v)}\\
        &=1+\Psi(k^{-1}:\eta)
    \end{align*} and similarly for the other inequality, as desired.
\end{proof}

For the next estimate, given $\Omega\subset (M,g)$ we recall the quantity \[D_{\Omega}=\sup\{\mathrm{diam}_{M}(W): \text{W is a connected component of $\Omega$}\}.\] Seeing as though $S_k=\mathrm{diam}_{g_k}(\mathbb S^n)\leqslant D$, we immediately obtain the following lemma: 

\begin{lem}[Estimating $D_{\Omega_{k, \eta}}$, $D_{\Omega_{rd,\eta}}$, and $a$]\label{DU}  
For all $k\geqslant 1$ large enough, \[D_{\Omega_{k, \eta}}\leqslant D \qquad\text{and}\qquad D_{\Omega_{rd,\eta}}\leqslant\pi.\] Therefore, the parameter $a$ in the statement of Theorem \ref{LakzianSormani} may be taken such that $  a\leqslant\Psi(k^{-1}).$
\end{lem} 
It may be worth noting that one can easily obtain sharper estimates for $D_{\Omega_{k,\eta}}$ and thus $a$ in the above by explicitly constructing curves in $(\mathbb S^n,g_k)$ between pairs of points $x,y\in\Omega_{k,\eta}$ and bounding their lengths using Lemmas \ref{vestimate} and \ref{festimate}, instead of cheaply using the uniform diameter bound. In this case, we would be able to ensure a choice of $a$ such that that $a\leqslant\Psi(k^{-1}, \eta)$ (recall from Subsection \ref{notation} that this means $\Psi\to0$ as $k\to\infty$ \emph{and} $\eta\searrow 0$). We carry such an argument out in the following: 
 
\begin{lem}[Estimating $\lambda$, $h$, and $\overline h$]\label{distortion} For all large enough $k\geqslant 1$, 
   \[\lambda_k:=\sup_{x,y\in \Omega_\eta} \left|d_{g_k}(x,y)-d_{g_{rd}}(x,y)\right|\leqslant \Psi(k^{-1}, \eta).\]
   Therefore, we also have $0\leqslant h\leqslant \Psi(k^{-1}, \eta)$ and $0\leqslant \overline h\leqslant\Psi(k^{-1}, \eta)$.
\end{lem}

\begin{proof}
Fix any $x,y\in\Omega_\eta$. We first prove that \[d_{g_k}(x,y)-d_{g_{rd}}(x,y)\leqslant \Psi(k^{-1}, \eta).\] To do so, let $\gamma$ be a minimizing $g_{rd}$ geodesic in $\mathbb S^n$ connecting $x$ to $y$, which may certainly leave $\Omega_{rd,\eta}$. Let $\tilde\gamma$ be the piecewise smooth curve from $x$ to $y$ contained in the closure of $\Omega_{rd,\eta}$ obtained by replacing the single connected portion of $\gamma$ outside $\Omega_{rd,\eta}$ with an intrinsically minimizing great circle arc in $\partial\Omega_{rd,\eta}$. This yields a piecewise smooth curve in $\Omega_\eta$ and thus in $\Omega_{k,\eta}$ which we continue to denote as $\tilde\gamma$. By the $C^0$ Cheeger--Gromov convergence of Lemma \ref{C0CheegerGromov} (applied on, say, $\Omega_{\eta/2}$), we have that \[d_{g_k}(x,y)\leqslant L_{g_k}(\tilde\gamma)\leqslant L_{g_{rd}}(\tilde\gamma)+\Psi(k^{-1})\leqslant L_{g_{rd}}(\gamma)+\Psi(\eta)+\Psi(k^{-1})=d_{g_{rd}}(x,y)+\Psi(k^{-1}, \eta).\] Next we prove the opposite inequality \[d_{g_{rd}}(x,y)-d_{g_k}(x,y)\leqslant \Psi(k^{-1}, \eta)\] by fixing a $g_k$ geodesic $\gamma: [0,1]\to \mathbb S^n$ from $x$ to $y$. This curve $\gamma$ may just as well leave $\Omega_{k, \eta}$, but since it must begin and end in $\Omega_{k,\eta}$ there is a maximal set of times of the particular form $\mathcal I:=[0,t_1)\cup(t_2, 1]\subset [0,1]$ so that $\gamma':=\gamma\vert_\mathcal{I}\subset\Omega_{k,\eta}$ \footnote{In fact, by using a cut and paste procedure as above aided by the generalization of Clairaut's Relation to general warped products (see eg. \cite{WarpedClairaut}), it is possible to show that a minimizing $g_k$ geodesic can have at most a single connected arc in each of the two connected components of $\mathbb S^n\setminus\Omega_{k,\eta}$.}. We simply replace the entire portion of $\gamma$ between $\gamma(t_1)$ and $\gamma(t_2)$ with an intrinsically minimizing great circle arc in $\partial\Omega_{rd,\eta}$ to similarly obtain a new curve $\tilde\gamma$ lying in the closure of $\Omega_{k,\eta}$. We therefore estimate as above that for all large $k\geqslant 1$, \begin{align*} d_{g_k}(x,y)=L_{g_k}(\gamma)\geqslant L_{g_k}(\gamma')\geqslant L_{g_{rd}}(\gamma')&-\Psi(k^{-1})\\ &\geqslant L_{g_{rd}}(\tilde\gamma)-\Psi(\eta)-\Psi(k^{-1})\geqslant d_{g_{rd}}(x,y)-\Psi(k^{-1}, \eta),\end{align*} where we have used the fact that the added portion in $\partial\Omega_\eta$ has round length less than $\pi\eta$. \end{proof}

Moving on, we establish convergence of the various volume quantities appearing in the estimates of Theorem \ref{LakzianSormani}: 

\begin{lem}[Volume Convergence]\label{volumeconv}
    For all $k\geqslant 1$ large enough, we have that 
    \begin{enumerate}
        \item $\lvert\mathrm{Vol}_{g_k}^n(\Omega_{k,\eta})-\mathrm{Vol}_{g_{rd}}^n(\Omega_{rd,\eta})\rvert\leqslant \Psi(k^{-1})$
        \item $\mathrm{Vol}^n_{g_k}(\mathbb S^n\setminus\Omega_{k,\eta})\leqslant \Psi(\eta).$ 
        \item $\mathrm{Vol}^{n-1}_{g_k}(\partial\Omega_{k,\eta})\leqslant\Psi(\eta)$.
    \end{enumerate}
        In particular, $\lvert \mathrm{Vol}^n_{g_k}(\mathbb S^n)-\mathrm{Vol}^n_{g_{rd}}(\mathbb S^n)\rvert\leqslant \Psi(k^{-1})$. 
\end{lem}

\begin{proof}Clearly the full volume convergence follows from (1) and (2) by taking $\eta>0$ arbitrarily small. (1) and (3) are implied directly by the $C^0$ Cheeger--Gromov convergence of Lemma \ref{C0CheegerGromov}, so it just remains to establish (2). To do so, we recall formula \ref{volformula} for the volume tensor \[d\mathrm{Vol}^n_{g_k}=\frac{r^{n-1}}{V_{k,-}(r)^{1/2}}d\mathcal{L}^1(r)\otimes d\mathrm{Vol}^{n-1}_{g_{rd}}\] valid on the open hemisphere $\mathbb S^n_-$, and analogously on $\mathbb S^n_+$. We may thus estimate \begin{align*} 
    \mathrm{Vol}^n_{g_k}\left([0,\eta]\times\mathbb S^{n-1}\right)&=n\omega_n\int_0^\eta\frac{r^{n-1}}{V_{k,-}(r)^{1/2}}dr\\ 
    &\leqslant n\omega_n\eta^{n-1}\int_0^\eta\frac{1}{V_{k,-}(r)^{1/2}}dr\\
    &\leqslant n\omega_n\eta^{n-1}D=\Psi(\eta).
\end{align*} Indeed, the integral in the penultimate line is the arc length of a segment of a meridian of $(\mathbb S^n, g_k)$ starting from a pole. Seeing as though every meridian gives a minimizing path from one pole to the other (while also realizing the diameter of $(\mathbb S^n, g_k)$--see the proof of Lemma 2.1 in \cite{PTW}), we arrive at the final bound. The other component of $\mathbb S^n\setminus\Omega_{k,\eta}$ enjoys an analogous estimate, so we may conclude. 
\end{proof} 

At last, putting together Lemmas \ref{C0CheegerGromov}, \ref{DU}, \ref{distortion}, \ref{volumeconv} and Theorem \ref{LakzianSormani} yields: \[\mathrm{d}_{\mathcal{VIF}}\left((\mathbb S^n, g_k),(\mathbb S^n, g_{rd})\right)\leqslant \Psi(k^{-1}, \eta).\] Taking $\eta\searrow 0$ arbitrarily small and sending $k\to\infty$ thereby contradicts the assumption that the manifolds $(\mathbb S^n, g_k)$ remain bounded away from $(\mathbb S^n, g_{rd})$ in the $\mathcal{VIF}$ distance, establishing Theorem \ref{main}. 
\end{proof}

\begin{proof}[Proof of Theorem \ref{GHSurgery}]
    As in the proof of Theorem \ref{main}, for the sake of contradiction we would begin with a fixed $\delta_0>0$ and a sequence of would-be counterexamples $(\mathbb S^n, g_k)$ with the added property that for no smooth $Z_k\subset\mathbb S^n$ with at most two connected components satisfying \[\mathrm{Vol}^n_{g_k}(Z_k)+\mathrm{Vol}^{n-1}_{g_k}(\partial Z_k)\leqslant \delta_0\] is it true that $\mathrm{d}_{GH}((\mathbb S^n\setminus Z_k, g_k), (\mathbb S^n, g_{rd}))<\delta_0$. Set $Z_{k,\eta}=\mathbb S^n\setminus\overline{\Omega_{k,\eta}}$ (See Figure \ref{fig:badset}). Since \[\mathrm{d}_\mathcal{H}\left((\mathbb S^n\setminus\overline{\Omega_{rd,\eta}}, g_{rd}),(\mathbb S^n, g_{rd})\right)\leqslant\Psi(\eta),\] Lakzian--Sormani's Theorem \ref{LakzianSormani} together with Lemmas \ref{C0CheegerGromov}, \ref{DU}, \ref{distortion}, \ref{volumeconv} yield \begin{align*}
        &\mathrm{d}_{GH}\left((\mathbb S^n\setminus Z_{k,\eta}, g_k),(\mathbb S^n, g_{rd})\right)\leqslant \Psi(k^{-1}, \eta)\\
        &\mathrm{Vol}^n_{g_k}(Z_{k,\eta})+\mathrm{Vol}_{g_k}^{n-1}(\partial Z_{k,\eta})\leqslant \Psi(\eta: k^{-1}).
    \end{align*} Taking $\eta\searrow 0$ arbitrarily small and sending $k\to\infty$ thereby contradicts the assumption that the manifolds $(\mathbb S^n\setminus Z_{k,\eta}, g_k)$ must remain bounded away from $(\mathbb S^n, g_{rd})$ in the Gromov--Hausdorff distance, establishing Theorem \ref{GHSurgery}. 
\end{proof}

\begin{figure}
    \centering
    \includegraphics[width=0.4\linewidth]{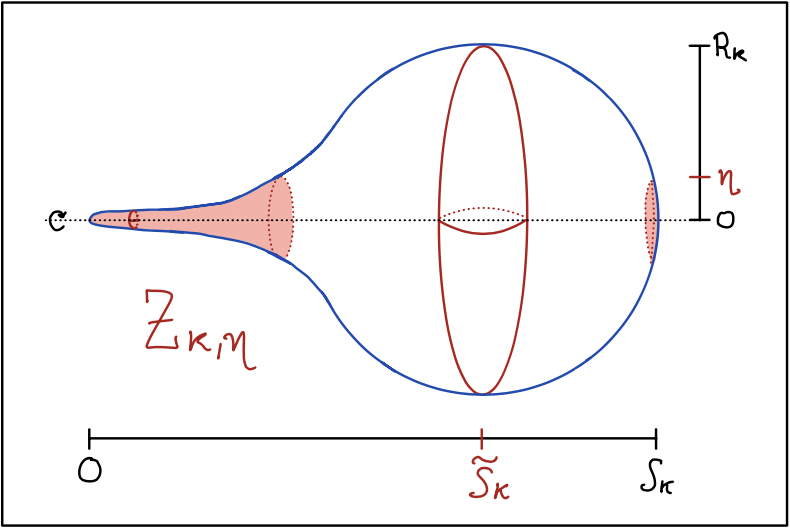}
    \caption{The ``bad set'' $Z_{k,\eta}$ which is surgically removed to obtain convergence.}
    \label{fig:badset}
\end{figure}

Finally, we establish the last result of our paper: 

\begin{proof}[Proof of Theorem \ref{GHRicci}]

We argue just as above under the new assumption of $\mathrm{Ric}\geqslant0$, where the key difference is the concavity of the warping functions $f_k(s)$ on $[0,D_k]$. This allows us to use Proposition \ref{diameter} to give an a priori diameter upper bound for the metrics $g_k$, removing the diameter assumption of Theorem \ref{GHSurgery}. This also allows us the rephrase the strong $\mathrm{MinA}$ condition as a condition on $W_g$.

To establish the result using what have so far (as in the proofs of Theorems \ref{main} and \ref{GHSurgery}), all that remains to be established is the following lemma: 

\begin{lem}[Hausdorff Convergence of the Subregions]\label{hausdorffsubregion} 
For all large enough $k\geqslant 1$, we have \[\mathrm{d}_\mathcal{H}((\Omega_{k,\eta},g_k), (\mathbb S^n, g_k))\leqslant\Psi(k^{-1},\eta) \qquad\text{and}\qquad \mathrm{d}_\mathcal{H}((\Omega_{rd,\eta},g_{rd}), (\mathbb S^n, g_{rd}))\leqslant \Psi(\eta). \]
\end{lem}

\begin{proof}
    The second estimate is obvious, so it only remains to consider the first. Seeing as though $(\Omega_{k,\eta},g_k)\subset(\mathbb S^n, g_k)$, it suffices to show that a $\Psi(k^{-1},\eta)$ open neighborhood of $(\Omega_{k,\eta},g_k)$ in $(\mathbb S^n, g_k)$ contains all of $(\mathbb S^n, g_k)$. 

    To show this, we consider $(s,\theta)$ coordinates for the complement of $\Omega_{k,\eta}$, where we recall that $s$ is the $g_k$ distance from the pole at $s=0$ to the point with coordinates $(s,\theta)$. Without loss of generality, let us show the estimate for the connected component of $\mathbb S^n\setminus\Omega_{k,\eta}$ written in coordinates as \[\left([0,s_k]\times\mathbb S^{n-1}, g_k=ds^2+f_k(s)^2g_{rd}\right),\] where $s_k$ is the unique parameter less than $\tilde S_k$ where $\eta=f_k(s_k)$. Seeing as though $f_k(0)=0$, $f_k(s_k)=\eta$, $f_k''(s)\leqslant 0$ on $[0, s_k]$, and $f_k'(s_k)=V_k(\eta)^{1/2}\to V_{rd}(\eta)^{1/2}=\sqrt{1-\eta^2}>0$, we see directly from integration that \[\eta=f_k(s_k)=\int_0^{s_k} f_k'(\xi)d\xi\geqslant s_kf_k'(s_k)\geqslant s_k\left(\sqrt{1-\eta^2}-\Psi(k^{-1})\right)\] for all large enough $k\geqslant 1$. Therefore, $0<s_k\leqslant \Psi(\eta: k^{-1})$, telling us that the diameter of each connected component of the ``missed'' region $\mathbb S^n\setminus\Omega_{k,\eta}$ can be made arbitrarily small by taking $k\geqslant 1$ large enough and sending $\eta\searrow 0$. The Hausdorff distance estimate thus follows.     
\end{proof} At long last, adding Lemma \ref{hausdorffsubregion} to the prior sm\"org\r{a}sbord of Lemmas \ref{C0CheegerGromov}, \ref{DU}, \ref{distortion}, \ref{volumeconv}, applying Lakzian--Sormani's Theorem \ref{LakzianSormani}, and taking $\eta\searrow 0$ yields \[\mathrm{d}_{GH}\left((\mathbb S^n, g_k),(\mathbb S^n, g_{rd})\right)\leqslant \Psi(k^{-1}).\] Sending $k\to\infty$ thereby contradicts the assumption that the manifolds $(\mathbb S^n, g_k)$ remain bounded away from $(\mathbb S^n, g_{rd})$ in the Gromov--Hausdorff distance, establishing Theorem \ref{GHRicci}.
\end{proof}

\printbibliography
\end{document}